\documentclass[a4paper, leqno, 11pt]{amsart}
\usepackage{amsfonts}
\usepackage{amssymb}
\usepackage{graphicx}
\graphicspath{ {CePC/Bureau/} }
\usepackage[applemac,utf8]{inputenc}
\usepackage{color}
\usepackage{hyperref}
\usepackage{mathdots}
\usepackage{float}
\newtheorem{lemma}{Lemma}[section]

\newtheorem{definition}[lemma]{Definition}
\newtheorem{proposition}[lemma]{Proposition}
\newtheorem{theorem}[lemma]{Theorem}
\newtheorem{remark}[lemma]{Remark}

\newtheorem{corollary}[lemma]{Corollary}

\title[Infinite products of nonnegative matrices]{Normalized image of a vector by an infinite product of nonnegative matrices}
\author[]{}
\address{Alain Thomas,
448 all\'ee des Cantons, 83640 Plan d'Aups Sainte Baume, France}
\email{alain-yves.thomas@laposte.net}

\date{July 2021}

\begin{document}

\maketitle

\centerline{\it Dedicated to the memory of Adrien Douady}

\begin{abstract}A sofic measure is the image of a Markov probability measure by a continuous morphism, and can be represented by means of products of matrices $A_n$ that belong to a finite set of nonnegative matrices. To prove that the multifractal formalism holds for such a measure, it is necessary to know whenever the sequence $n\mapsto\frac{A_1\cdots A_nv}{\Vert A_1\cdots A_nv\Vert}$ converges when $v$ is a positive vector. We give a sufficient condition for this convergence, that we use for the study of one Bernoulli convolution.\end{abstract}

\section{Introduction}

Let $\mathcal A=(A_n)_{n\in\mathbb N}$ be a sequence of nonnegative $d\times d$ matrices, meaning that the entries of $A_n$ are nonnegative, and let $P_n:=A_1\cdots A_n$. More generally, $P_{m,n}:=A_{m+1}\cdots A_n$. Since we consider any square matrix $M$ as a map $v\mapsto Mv$, we call "normalized image of the column-vector $v$ by the infinite product of the matrices $A_n$" the vector $v_{\mathcal A}:=\lim_{n\to\infty}\frac{P_nv}{\Vert P_nv\Vert}$, if existing. The sequence $n\mapsto\frac{P_nv}{\Vert P_nv\Vert}$ often converges while $n\mapsto\frac{P_n}{\Vert P_n\Vert}$ often diverges. To give an example let $(A_n)_{n\in\mathbb N}\in\{A,B\}^{\mathbb N}$, not eventually constant, where $A$ and $B$ are positive (that is, the entries of $A$ and $B$ are positive) and do not have a common left-eigenvector. Then for any nonnegative $v\ne0$ the sequence $n\mapsto\frac{P_nv}{\Vert P_nv\Vert}$ converges and $n\mapsto\frac{P_n}{\Vert P_n\Vert}$ diverges. 

The three sections are relatively independent. In Section~\ref{tsf} we give some precisions about both sequences and in Section~\ref{limpoints} we specify in what sense the sequence of matrices $n\mapsto\frac{P_n}{\Vert P_n\Vert}$ in general diverges. In Section \ref{sm} we indicate how the convergence of $n\mapsto\frac{P_nv}{\Vert P_nv\Vert}$ can be used to study the linearly representable measures, for instance the Bernoulli convolutions in Pisot base with finite R\'enyi expansion \cite{Ren}, \cite[\S2.4]{OST05}. We detail in Appendix A how to use Section 1, for one example of Bernoulli convolution.

See for instance \cite{BL85,Bo,DL92,DL01,EF97,F,Go,Koz13,L89} for various results about the infinite products of matrices. In \cite{OT13} (resp. \cite{T}) there is a necessary and sufficient condition on the finite sets $\mathcal M$ of $2\times2$ nonnegative matrices, for the sequences $n\mapsto\frac{P_nv}{\Vert P_nv\Vert}$ to converge uniformly (resp. pointwise) on $\mathcal M^\mathbb N$. See \cite{BO,BP,Dh,Eri} for the sofic measures and \cite{Erd39,PSS} for the Bernoulli convolutions. For the multifractal analysis and Lyapunov exponents, see \cite{FeandLau02,Fen03,Fen03bis,Fen04,Fen09,FH10,FO,Led84}.

We use preferably the norm $\Vert X\Vert:=$ the sum of the absolute values of the entries of $X$ (matrix or vector), but for any couple of norm $(N_1,N_2)$ the convergence of a sequence $n\mapsto\frac{X_n}{N_1(X_n)}$ is equivalent to the one of $n\mapsto\frac{X_n}{N_2(X_n)}$ because $\frac{X_n}{N_2(X_n)}=\frac{X_n/N_1(X_n)}{N_2(X_n/N_1(X_n))}$. We call $\mathcal I(X)$ the set of the indices of the nonnull entries in $X$, and $\mathcal Z(X)$ the $(0,1)$-matrix or vector such that $\mathcal I(\mathcal Z(M))=\mathcal I(M)$. We denote by $u_1,\dots,u_d$ the canonical column-vectors and by $u$ their sum. The shift $\sigma$ is defined by $\sigma((x_n)_{n\in\mathbb N})=(x_{n+1})_{n\in\mathbb N}$.

{\bf Acknowledgement. --} {\it This paper was written in collaboration with Eric Olivier, in particular the second section about multifractal analysis.}

\section{The sequence $n\mapsto\frac{P_nv}{\Vert P_nv\Vert}$}\label{tsf}

\subsection{The case of $2\times2$ matrices}\label{2}According to \cite{T}, the only nontrivial cases of divergence of $n\mapsto\frac{P_nv}{\Vert P_nv\Vert}$ when the matrices $A_n$ belong to a finite set $\mathcal M$ of $2\times2$ nonnegative matrices, are
the cases

$\mathcal M\ni\left(\begin{smallmatrix}a&b\\c&0\end{smallmatrix}\right),\left(\begin{smallmatrix}a'&0\\0&d'\end{smallmatrix}\right)$ with $abc\ne0<d'-a'$, and 

$\mathcal M\ni\left(\begin{smallmatrix}0&b\\c&d\end{smallmatrix}\right),\left(\begin{smallmatrix}a'&0\\0&d'\end{smallmatrix}\right)$ with $bcd\ne0<a'-d'$.

Consider for instance the product\begin{equation}\label{continuedfraction}Q_k:=\left(\begin{smallmatrix}\frac12&0\\0&1\end{smallmatrix}\right)^{n_1}\left(\begin{smallmatrix}1&1\\1&0\end{smallmatrix}\right)\cdots\left(\begin{smallmatrix}\frac12&0\\0&1\end{smallmatrix}\right)^{n_k}\left(\begin{smallmatrix}1&1\\1&0\end{smallmatrix}\right)=:\left(\begin{smallmatrix}p_k&2^{-n_k}p_{k-1}\\q_k&2^{-n_k}q_{k-1}\end{smallmatrix}\right)\end{equation}with $\left\{\begin{array}{l}p_k=2^{-n_k}p_{k-1}+2^{-n_{k-1}}p_{k-2}\\q_k=2^{-n_k}q_{k-1}+2^{-n_{k-1}}q_{k-2}\end{array}\right.$ (see\cite{Per}).

For any $k$, if $n_1,\dots,n_{k-1}$ are known, the limits when $n_k\to\infty$ of the differences $\frac1{p_k+q_k}\left(\begin{smallmatrix}p_k\\q_k\end{smallmatrix}\right)-\frac1{p_{k-2}+q_{k-2}}\left(\begin{smallmatrix}p_{k-2}\\q_{k-2}\end{smallmatrix}\right)$ and $\frac{Q_kv}{\Vert Q_kv\Vert}-\frac1{p_{k-i}+q_{k-i}}\left(\begin{smallmatrix}p_{k-i}\\q_{k-i}\end{smallmatrix}\right)$ with $i=1$ if $v_1=0$ or $i=0$ otherwise, are null. So, choosing $n_k$ large enough they have norm $\le2^{-k}\varepsilon$. Consequently $\frac{Q_kv}{\Vert Q_kv\Vert}$ is approximately $\frac1{p_1+q_1}\left(\begin{smallmatrix}p_1\\q_1\end{smallmatrix}\right)$ or $\frac1{p_2+q_2}\left(\begin{smallmatrix}p_2\\q_2\end{smallmatrix}\right)$ and the error term has norm at most~$\varepsilon$. Choosing $\varepsilon$ small enough, the sequence $n\mapsto\frac{Q_kv}{\Vert Q_kv\Vert}$ diverges because (\ref{continuedfraction}) implies $\left(\begin{smallmatrix}p_1\\q_1\end{smallmatrix}\right),\left(\begin{smallmatrix}p_2\\q_2\end{smallmatrix}\right)$ non-collinear.

\subsection{Rank-one properties}The previous subsection gives an example of product $P_n$ of matrices belonging to a set of two matrices, that cannot be approximated by a rank-one matrix (a matrix $R_n=c_nr_n$ where $c_n$ is a column-vector and $r_n$ a row-vector). Because, looking at the matrices $Q_{k,h}:=Q_k\left(\begin{smallmatrix}\frac1{2^h}&0\\0&1\end{smallmatrix}\right)$ for $h=0,1,\dots,n_{k+1}$, the ratio $\frac{\Vert Q_{k,h}u_1\Vert}{\Vert Q_{k,h}u_2\Vert}$ is $>1$ for $h=0$ and $<1$ for $h=n_{k+1}$, implying $\exists h(k)\ \frac{\Vert Q_{k,h(k)}u_1\Vert}{\Vert Q_{k,h(k)}u_2\Vert}\in[1,2]$, while the limit-points of the normalized columns of $Q_{k,h(k)}$ are non-collinear.

\begin{theorem}\label{rankone}(i) For any sequence $(M_n)_{n\in\mathbb N}$ of complex-valued $d\times d$ matrices with singular values $\delta_1(n)\ge\delta_2(n)\ge\dots$, one has the equivalences$$\begin{array}{l}\lim_{n\to\infty}\frac{\delta_2(n)}{\delta_1(n)}=0\ \Leftrightarrow\ \exists (c_n,r_n)\ \lim_{n\to\infty}\big(\frac{M_n}{\Vert M_n\Vert}-c_nr_n\big)=0;\\\exists\lim_{n\to\infty}\frac{M_nv}{\Vert M_nv\Vert}\text{ indep. of }v>0\Leftrightarrow\exists (c,r_n)\ \lim_{n\to\infty}\big(\frac{M_n}{\Vert M_n\Vert}-cr_n\big)=0.\end{array}$$

(ii) Let $\mathcal A=(A)_{n\in\mathbb N}$ with $A_n$ nonnegative. The limit of $n\mapsto\frac{P_nv}{\Vert P_nv\Vert}$ exists and is constant on the set $\mathcal V_\mathcal A:=\{v\ge0\;;\;\inf_n\frac{\Vert P_nv\Vert}{\Vert P_n\Vert}>0\}$ if there exists a limit-point $P^{(s)}:=\lim_{k\to\infty}\frac{P_{n_k}}{\Vert P_{n_k}\Vert}$ ($s=(n_k)_{k\in\mathbb N}$ increasing) such that

$*$ $\exists K\in\mathbb N\ \inf_{k\ge K}\inf_{n_{k+1}\le n<n_{k+1}}\frac{\Vert P^{(s)}P_{n_k,n}\Vert}{\Vert P_{n_k,n}\Vert}>0$,

$*$ for any limit-point $Q$ of the sequence $k\mapsto\frac{P_{n_k,n_{k+1}}}{\Vert P_{n_k,n_{k+1}}\Vert}$ there does not exist a permutation matrix $P$ such that $P^{-1}QP=\left(\begin{smallmatrix}A&0&B\\0&C&D\\0&0&E\end{smallmatrix}\right)$ with $A$ and $C$ column-allowable (that is, without null column).\end{theorem}

\begin{proof}[\bf Proof](i) We write $M_n=S_n\left(\begin{smallmatrix}\delta_1(n)&\dots&0\\\vdots&\ddots&\vdots\\0&\dots&\delta_d(n)\end{smallmatrix}\right){^t\overline{T_n}}$ with $S_n$ and $T_n$ unitary matrices. If $\lim_{n\to\infty}\frac{\delta_2(n)}{\delta_1(n)}=0$, the $\delta_i(n), i\ge2$, are in $o(\Vert M_n\Vert)$ and consequently $\frac{M_n}{\Vert M_n\Vert}=(\frac{\delta_1(n)}{\Vert M_n\Vert}S_nu_1)({^tu_1}{^t\overline{T_n}})+o(1)=:c_nr_n+o(1)$.

Suppose now that $\frac{\delta_2(n)}{\delta_1(n)}$ does not converge to $0$. There exist an increasing sequence $(n_k)_{k\in\mathbb N}$ such that $\lim_{k\to\infty}\frac{M_{n_k}}{\Vert M_{n_k}\Vert}=S\Delta{^t\overline{T}}$, where $S,T$ are unitary and the two first diagonal entries of the diagonal matrix $\Delta$ are positive. Consequently rank$(S\Delta{^t\overline{T}})$=rank$(\Delta)\ge2$. The vectors $c_n,r_n$ such that $\lim_{n\to\infty}\Big(\frac{M_n}{\Vert M_n\Vert}-c_nr_n\Big)=0$ do not exist: this should imply that $c_{n_k}r_{n_k}$ (of rank $1$) tends to $S\Delta{^t\overline{T}}$ (of rank $\ge2$).

To prove the second equivalence, suppose $\exists w=\lim_{n\to\infty}\frac{M_nv}{\Vert M_nv\Vert}$ independent of $v>0$. Let $u_{j,k}:=u_j+\frac1ku$. There exists $n_k$ such that $n\ge n_k\Rightarrow\big\Vert \frac{M_nu_{j,k}}{\Vert M_nu_{j,k}\Vert}-w\big\Vert\le\frac1k$. Using the classical inequality $\big\Vert\frac x{\Vert x\Vert}-\frac y{\Vert y\Vert}\big\Vert\le\frac{2\Vert x-y\Vert}{\max(\Vert x\Vert,\Vert y\Vert)}$ we deduce $\big\Vert \frac{M_nu_j}{\Vert M_nu_j\Vert}-w\big\Vert\le\frac{2\Vert M_n\Vert}{k\Vert M_nu_j\Vert}+\frac1k$. We multiply by $\frac{\Vert M_nu_j\Vert}{\Vert M_n\Vert}$ and conclude that $\lim_{n\to\infty}\big(\frac{M_n}{\Vert M_n\Vert}-cr_n\big)=0$ with $c=w$ and $r_n$ of entries $\frac{\Vert M_nu_j\Vert}{\Vert M_n\Vert}$.

Conversely suppose $\exists(c,r_n)\ \lim_{n\to\infty}\big(\frac{M_n}{\Vert M_n\Vert}-cr_n\big)=0$. We multiply at the right by $\frac{\Vert M_n\Vert}{\Vert M_nv\Vert}v$ (bounded vector) and deduce $\lim_{n\to\infty}\big\Vert\frac{M_nv}{\Vert M_nv\Vert}-\lambda_nc\big\Vert=0$ with $\lambda_n=\frac{\Vert M_n\Vert}{\Vert M_nv\Vert}r_nv$ . This implies $\Vert\lambda_n c\Vert\to1$, $\lambda_n\to\frac1{\Vert c\Vert}$ and $\frac{M_nv}{\Vert M_nv\Vert}\to\frac c{\Vert c\Vert}$.

(ii) We denote by $Q_n$ the matrix $\frac{P_{n_k,n}}{\Vert P_{n_k
,n}\Vert}$, where $k$ is the integer such that $n_{k+1}\le n<n_{k+2}$. The first condition of (ii) is equivalent to say that any convergent subsequence of $n\mapsto\Vert P^{(s)}Q_n\Vert$ has a positive limit, or to say that $P^{(s)}Q^{(s,s')}\ne0$ for any limit-point $Q^{(s,s')}:=\lim_{i\to\infty}Q_{n'_i}$ ($s'=(n'_i)_{i\in\mathbb N}$ increasing). For such a sequence $s'$, let us prove that\begin{equation}\label{PQ}\exists P^{(s')}:=\lim_{i\to\infty}\frac{P_{n'_i}}{\Vert P_{n'_i}\Vert}\text{ and }P^{(s')}=\frac{P{(s)}Q^{(s,s')}}{\Vert P{(s)}Q^{(s,s')}\Vert}\end{equation}This is a consequence of the following lemma, because $P_{n'_i}=P_{n_k}P_{n_k,n'_i}$ with $n_{k+1}\le n'_i<n_{k+2}$.

\begin{lemma}\label{abovernormab}If $\lim_{n\to\infty}\frac{A_n}{\Vert A_n\Vert}=A$, $\lim_{n\to\infty}\frac{B_n}{\Vert B_n\Vert}=B$ and $AB\ne0$, then there exists $\lim_{n\to\infty}\frac{A_nB_n}{\Vert A_nB_n\Vert}=\frac{AB}{\Vert AB\Vert}$.\end{lemma}

\begin{proof}[\bf Proof]It is simply because $\lim_{n\to\infty}\frac{A_n}{\Vert A_n\Vert}\frac{B_n}{\Vert B_n\Vert}=AB$ implies that

$\lim_{n\to\infty}\frac{\Vert A_nB_n\Vert}{\Vert A_n\Vert\Vert B_n\Vert}$ exists and is equal to $\Vert AB\Vert$.\end{proof}

Suppose there exist two distinct extreme points $c_1,c_2$ of the convex hull $\mathcal C$ of the normalized vectors $\frac{P^{(s)}u_j}{\Vert P^{(s)}u_j\Vert}$ ($j$ such that $P^{(s)}u_j\ne0$). From the definition of the extreme points, $c_1,c_2$ cannot be positive linear combinations of several vectors $\frac{P^{(s)}u_j}{\Vert P^{(s)}u_j\Vert}$, hence there exist two nonempty sets $J_1,J_2$ such that $j\in J_1\Leftrightarrow\frac{P^{(s)}u_j}{\Vert P^{(s)}u_j\Vert}=c_1$ and $j\in J_2\Leftrightarrow\frac{P^{(s)}u_j}{\Vert P^{(s)}u_j\Vert}=c_2$.

By compacity there exists a limit-point $Q$ of the sequence $k\mapsto\frac{P_{n_k,n_{k+1}}}{\Vert P_{n_k,n_{k+1}}\Vert}$. By (\ref{PQ}), $P^{(s)}=\frac{P^{(s)}Q}{\Vert P^{(s)}Q\Vert}$, implying that the columns of $Q$ with index in $J_1$ (resp. $J_2$) have null entries except the one indexed by $J_1$ (resp. $J_2$). This contradicts the second hypothesis of~(ii). So $\mathcal C$ has only one extreme point~$c$ and, by the Krein-Milman theorem, $\mathcal C$ is the convex hull of its extreme points, $\mathcal C=\{c\}$ and $P^{(s)}=cr$.

For any couple of vectors $v,w\in\mathcal V_\mathcal A$, we have to prove that any limit-point $\lim_{k\to\infty}\frac{P_{m_i}v}{\Vert P_{m_i}v\Vert}$ is equal to $\lim_{k\to\infty}\frac{P_{n_k}w}{\Vert P_{n_k}w\Vert}$. By compacity we can replace $t:=(m_i)_{i\in\mathbb N}$ by a subsequence, for $i\mapsto Q_{m_i}$ to converge. As seen, the limit-point $\lim_{i\to\infty}\frac{P_{m_i}}{\Vert P_{m_i}\Vert}$ exists and is equal to $\frac{P^{(s)}Q^{(s,t)}}{\Vert P^{(s)}Q^{(s,t)}\Vert}$. Since $v,w\in\mathcal V_\mathcal A$, $P^{(s)}Q^{(s,t)}v$ and $P^{(s)}w$ are nonnull. $\frac{P_{m_i}v}{\Vert P_{m_i}v\Vert}=\frac{(P_{m_i}/\Vert P_{m_i}\Vert)v}{\Vert(P_{m_i}/\Vert P_{m_i}\Vert)v\Vert}$ tends to $\frac{P^{(s)}Q^{(s,t)}v}{\Vert P^{(s)}Q^{(s,t)}v\Vert}$ and $\frac{P_{n_k}w}{\Vert P_{n_k}w\Vert}=\frac{(P_{n_k}/\Vert P_{n_k}\Vert)w}{\Vert(P_{n_k}/\Vert P_{n_k}\Vert)w\Vert}$ tends to $\frac{P^{(s)}w}{\Vert P^{(s)}w\Vert}$; both limits are equal to $c$.\end{proof}

\begin{remark}In the first item of Theorem \ref{rankone}, the assertions of the first equivalence are equivalent to "all the convergent subsequences of $n\mapsto\frac{M_n}{\Vert M_n\Vert}$ converge to a rank-one matrix", because the set of the rank-one normalized matrices is closed. The assertions of the second equivalence imply that $c$ is colinear to $\lim_{n\to\infty}\frac{M_nv}{\Vert M_nv\Vert}$ and to the limit of the first singular vector of $M_n$.

In the counterexample $M_n=\prod_{k=1}^n\left(\begin{smallmatrix}1&0&1\\0&1&0\\0&0&1\end{smallmatrix}\right)^{2^k}\left(\begin{smallmatrix}1&0&0\\0&1&1\\0&0&1\end{smallmatrix}\right)^{2^k}$, $c_n$ is non-constant.\end{remark}

\begin{remark}The second item of Theorem \ref{rankone} is useful when it is not possible to compute $P_n$ nor the limit-point $P^{(s)}$ but possible to have information about $\mathcal Z(P^{(s)})$, see the example of Appendix~\ref{example}. On the other hand, in the example of Appendix \ref{3} one can compute $P_n$.

In the counterexample $\forall n\ A_n=\left(\begin{smallmatrix}1&0&1\\0&1&1\\0&0&\frac12\end{smallmatrix}\right)$, the second condition of Theorem~\ref{rankone}(ii) is not satisfied and $\lim_{n\to\infty}\frac{P_nv}{\Vert P_nv\Vert}$ depends on $v$.\end{remark}

\section{Sofic measures, multifractal analysis}\label{sm}

\subsection{Sofic and linearly representable measures (i.e. representable by products of matrices)}\label{Mslr}We give in \cite[Definition 5]{T2} a suitable definition of both notions to prove that they are equivalent \cite[Theorem 7]{T2}. This result is known, we use the methods of \cite[Theorem 4.28]{BP} and \cite{Eri}.

\subsection{Multifractal formalism and weak-Gibbs property}

Both notions are related, and the weak-Gibbs property of a sofic measure is related to the uniform convergence of a sequence of the form $n\mapsto\frac{P_nv}{\Vert P_nv\Vert}$ (Theorem \ref{weak-Gibbs}).

The {\it multifractal analysis} \cite{G,HP,Ols,Pes} is a particular way of analysing the local structure of measures. Let $\mu$ be a probability-measure on $[0,1]$; its {\it singularity spectrum} $\tau_{\rm sing}:\mathbb R\to[-\infty,1]$ is defined by
$$
\begin{array}{lcl}\tau_{\rm sing}(\alpha)&:=&\text{H-dim}(E_\mu(\alpha))\quad(\text{Hausdorff dimension of }E_\mu(\alpha)),\text{ where}\\
E_\mu(\alpha)&:=&\{x\;;\;\dim_\mu(x)\text{ exists and }\dim_\mu(x)=\alpha\}\quad(\text{level-set of }\mu),\\
\dim_\mu(x)&:=&\lim_{r\to0}\log(\mu([x-r,x+r]))/\log r\quad(\text{local dimension of }\mu)\end{array}
$$
(by convention, the Hausdorff dimension of the empty set is $-\infty$). Its {\it scale spectrum} (or {\it$L^q$-spectrum}) $\tau_{\rm scale}:\mathbb R\to[-\infty,1]$ is defined by
$$
\tau_{\rm scale}(q):=\liminf_{r\to0}\Big(\log_r\Big(\inf_\mathcal I\Big(\sum_{k=1}^n(\mu(I_k))^q\Big)\Big)\Big),
$$
where $\mathcal I$ is the set of the covers of the support of $\mu$ by closed intervals of length $r$.

The scale spectrum is easier to compute or to approximate than the singularity spectrum. The multifractal formalism, defined as follows, enables to compute the second when one knows the first.\begin{definition}One says that $\mu$ satisfies the {\it multifractal formalism} if the singularity spectrum is the Legendre-transform conjugate of the scale spectrum (see for instance \cite{HJKPS}) in the sense that, for any $\alpha\in\mathbb R$,$$
\tau_{\rm sing}(\alpha)=\inf_{q\in\mathbb R}(\alpha q-\tau_{\rm scale}(q)).
$$\end{definition}

According to \cite[Theorem A']{FO}, the {\it weak-Gibbs measures} defined by Yuri \cite[\S5]{Y} as follows, satisfy the multifractal formalism.

\begin{definition}Let $\mathcal S$ be a system of affine contractions $S_k:[0,1]\mapsto[0,1]$ for
$k\in\{0,\dots,b-1\}$ that satisfy the condition $[0,1)=\bigcup_{k\in\{0,\dots,b-1\}}S_k([0,1))$ (disjoint union).
We define the basic subinterval $[\omega_1\dots\omega_n]_{\mathcal S}\subset[0,1)$, for any word $\omega_1\dots\omega_n\in\{0,\dots,b-1\}^n$, by
$$
[\omega_1\dots\omega_n]_{\mathcal S}:=S_{\omega_1}\circ \cdots\circ S_{\omega_n}([0,1)),
$$
and we say that a probability-measure $\mu$ on $[0,1]$ has the {\sl weak-Gibbs property with respect to $\mathcal S$} if there exists a continuous map $\Phi:\Omega=\{0,\dots,b-1\}^\mathbb N\to\mathbb R$, called a potential of $\mu$, such that
\begin{equation}\label{defGibbs}
\lim_{n\to\infty}\Big(\frac{\mu([\omega_1\dots\omega_n]_{\mathcal S})}{\exp\big(\sum_{k=0}^{n-1}\Phi(\sigma^k\omega\big)}\Big)^{\frac1n}=1\text{ uniformly on }\Omega\ (\sigma\omega:=(\omega_{n+1})_{n\in\mathbb N}).
\end{equation}
\end{definition}

Below is a method to prove the weak-Gibbs property for a sofic measure:

\begin{theorem}\label{weak-Gibbs}
Let $\mu$ be a probability-measure on $[0,1]$, $\mathcal S$ a system of affine contractions \text{$S_k:[0,1]\mapsto[0,1]$} for $k\in\{0,\dots,b-1\}$, and let $\nu$ be the probability-measure defined on $\Omega=\{0,\dots,b-1\}^\mathbb N$ by setting, for any cylinder set $[\omega_1\dots\omega_n]$,
$$
\nu([\omega_1\dots\omega_n])=\mu([\omega_1\dots\omega_n]_\mathcal S).
$$
We suppose that $\nu$ is sofic. Using the linear representation of the sofic measures by matrices $M_i$, row-vectors $r_i$ and column-vector $c$ \cite[Definition 5]{T2}, we suppose that the sequence of column-vectors $n\mapsto c_{\omega,n}:=\frac{M_{\omega_1}\cdots M_{\omega_n}c}{\Vert M_{\omega_1}\cdots M_{\omega_n}c\Vert}$ converges uniformly on $\Omega$ to a limit~$c_\omega$. We also suppose
\begin{align}
&\forall i\in\{0,\dots,b-1\},\ (r_ic_{\omega,n})^{\frac1n}\to1\text{ uniformly in }\omega\in\Omega,\label{randc}\\&\forall\omega\in\Omega,\ M_ic_\omega\ne0.\label{Mandc}
\end{align}
Then $\mu$ is weak-Gibbs with respect to $\mathcal S$ and to the potential $\Phi$ defined by $\Phi(\omega):=\log\Vert M_{\omega_1}c_{\sigma\omega}\Vert$.
\end{theorem}

\begin{proof}[{\bf Proof}]We compute
\begin{equation}\label{muandPhi}\begin{array}{rcl}
\frac{\mu([\omega_1\dots\omega_n]_{\mathcal S})}{\Vert c\Vert}&=&\frac{r_{\omega_1}M_{\omega_2}\cdots M_{\omega_n}c}{\Vert M_{\omega_2}\cdots M_{\omega_n}c\Vert}\cdot\frac{\Vert M_{\omega_2}\cdots M_{\omega_n}c\Vert}{\Vert M_{\omega_3}\cdots M_{\omega_n}c\Vert}\cdots\frac{\Vert M_{\omega_n}c\Vert}{\Vert c\Vert}\\&=&r_{\omega_1}c_{\sigma\omega,n-1}\cdot\Vert M_{\omega_2}c_{\sigma^2\omega,n-2}\Vert\cdots\Vert M_{\omega_n}c_{\sigma^n\omega,0}\Vert\end{array}
\end{equation}
and
\begin{equation}\label{Phiandmu}
\exp\big(\sum_{k=0}^{n-1}\Phi(\sigma^k\omega)\big)=\Vert M_{\omega_1}c_{\sigma\omega}\Vert\cdots\Vert M_{\omega_n}c_{\sigma^n\omega}\Vert.
\end{equation}
By (\ref{randc}), $(r_{\omega_1}c_{\sigma\omega,n-1})^{\frac1n}$ converges uniformly to $1$, as well as $\Vert M_{\omega_1}c_{\sigma\omega}\Vert^{\frac1n}$ because (\ref{Mandc}) and the continuity of $\omega\mapsto c_\omega$ imply $\forall\omega\ \frac1K\le\Vert M_{\omega_1}c_{\sigma\omega}\Vert\le K$ with $K$ constant. Since $c_{\omega,n}$ converges uniformly to $c_\omega$, there exists $\varepsilon_n\to0$ such that, for any $\omega\in\Omega$,
$$
1-\varepsilon_n\le\frac{\Vert M_{\omega_1}c_{\sigma\omega,n-1}\Vert}{\Vert M_{\omega_1}c_{\sigma\omega}\Vert}\le1+\varepsilon_n\ .
$$
So (\ref{muandPhi}) and (\ref{Phiandmu}) imply
$$\begin{array}{rcl}\frac{r_{\omega_1}c_{\sigma\omega,n-1}}{\Vert M_{\omega_1}c_{\sigma\omega}\Vert}(1-\varepsilon_{n-1})\cdots(1-\varepsilon_1)&\le&\frac{\mu([\omega_1\dots\omega_n]_{\mathcal S})/\Vert c\Vert}{\exp\big(\sum_{k=0}^{n-1}\Phi(\sigma^k\omega)\big)}\\&\le&\frac{r_{\omega_1}c_{\sigma\omega,n-1}}{\Vert M_{\omega_1}c_{\sigma\omega}\Vert}\cdot(1+\varepsilon_{n-1})\cdots(1+\varepsilon_1).\end{array}$$
Since $\lim_{n\to\infty}\big(\prod_{k=1}^{n-1}(1\pm\varepsilon_k)\big)^{\frac1n}=1$, $\mu$ satisfies (\ref{defGibbs}).\end{proof}

\subsection{Application to a Bernoulli convolution}\label{application}

From now $\beta\approx1.755$ is the unique real solution of the equation $\beta^3=2\beta^2-\beta+1$, and $\mu_\beta$ is the Bernoulli convolution (see \cite{Erd39,PSS}) associated to $\beta$:
\begin{equation*}\label{ICBMeq}
\mu_\beta:=
\text{\huge $*$}_{n=1}^{\infty}\Big(\frac12\delta_0+\frac12\delta_{(\beta-1)\beta^{-n}}\Big).
\end{equation*}

We choose this value of $\beta$ because the involved matrices have order $7$, while the matrices are very larger in the case $\beta^3=\beta^2+1$ for instance, or small when $\beta$ is multinacci \cite{OST05} or  quadratic.

\begin{lemma}\label{SiM}
The measure $\mu_\beta$ has support $[0,1]$. Setting$$\begin{array}{l}S_0(x):=\frac x\beta, S_1(x):=\frac1\beta+\frac x{\beta^2},S_2(x):=\frac1\beta+\frac1{\beta^2}+\frac x{\beta^4},i_0:=0,i_1:=\frac1\beta,\\i_2:=\frac1\beta-\frac1{\beta^2},i_3:=-\frac1{\beta^2},i_4:=1-\frac1\beta,i_5:=1-\frac1{\beta^2},i_6:=1-\frac1\beta+\frac1{\beta^2},\\2M_0:=\left(\begin{smallmatrix}1&0&0&0&0&0&0\\0&0&1&0&0&0&0\\0&0&0&1&1&0&0\\0&0&0&0&0&0&0\\1&0&0&0&0&0&1\\0&0&0&0&1&0&0\\0&1&0&0&0&0&0\end{smallmatrix}\right)\ 4M_1:=\left(\begin{smallmatrix}0&0&1&1&0&0&0\\0&0&0&0&0&1&0\\0&0&0&1&1&0&0\\1&0&0&0&0&0&0\\0&0&1&0&0&0&0\\0&0&0&0&0&0&0\\0&0&0&0&0&0&0\end{smallmatrix}\right)\ 16M_2:=\left(\begin{smallmatrix}1&0&0&0&1&0&1\\0&0&0&0&0&0&0\\1&0&0&0&0&0&1\\0&0&0&1&1&0&0\\0&0&0&0&1&0&0\\0&0&0&0&0&0&0\\0&0&0&0&0&0&0\end{smallmatrix}\right)\\20\ ^tc:=\left(\begin{smallmatrix}12&8&13&4&12&6&4\end{smallmatrix}\right)\end{array}
$$
one has for any basic interval $[\omega_1\dots\omega_n]_{\mathcal S}=S_{\omega_1}\circ \cdots\circ S_{\omega_n}([0,1))$
\begin{equation}\label{translatedcylinders}
\left(\begin{smallmatrix}\mu_\beta\big(i_0+\frac1{\beta}[\omega_1\dots\omega_n]_{\mathcal S}\big)\\\vdots\\\mu_\beta\big(i_6+\frac1{\beta}[\omega_1\dots\omega_n]_{\mathcal S}\big)\end{smallmatrix}\right)=M_{\omega_1}\cdots M_{\omega_n}c.
\end{equation}
\end{lemma}

\begin{proof}[{\bf Proof}]Denoting by $P$ the uniform probability-measure defined on $\{0,1\}^\mathbb N$ by $P([\varepsilon_1\dots\varepsilon_n])=\frac1{2^n}$ we have, for any interval $I\subset \mathbb R$,
$$
\mu_\beta(I)=P\Big(\Big\{(\varepsilon_n)_{n\in\mathbb N}\in\{0,1\}^\mathbb N\;;\;(\beta-1)\sum_{n=1}^\infty\varepsilon_n\beta^{-n}\in I\Big\}\Big),
$$
and the support of $\mu$ is included in $[0,1]$ because  $0\le(\beta-1)\sum_{n=1}^\infty\varepsilon_n\beta^{-n}\le1$. 

Let $k\in\{0,1,2\}$, we first determine a relation between $\mu_\beta\big(\gamma+\frac1\beta S_k(I)\big)$, where $\gamma\in\mathbb R$, and $\mu_\beta\big(\gamma'+\frac1\beta I\big)$, where $\gamma'$ belongs to a finite set depending on $\gamma$. For $k=0$, $\mu_\beta\big(\gamma+\frac1\beta S_0(I)\big)$ is the probability that $(\beta-1)\big(\frac{\varepsilon_1}\beta+\frac{\varepsilon_2}{\beta^2}+\dots\big)\in\gamma+\frac1{\beta^2}I$, and this is equivalent to $(\beta-1)\big(\frac{\varepsilon_2}\beta+\frac{\varepsilon_3}{\beta^2}+\dots\big)\in\gamma_0+\frac1\beta I\quad\text{with}\quad\gamma_0=\beta\gamma-(\beta-1)\varepsilon_1$~. Since \text{$\varepsilon_1=$ $0$ or $1$} with probability $\frac12$~,
\begin{equation}\label{firstcase}
\begin{array}{l}\mu_\beta\big(\gamma+\frac1\beta S_0(I)\big)=\frac12\sum_{\gamma_0\in\Gamma_0}\mu_\beta\big(\gamma_0+\frac1\beta I\big),\text{ where}\\
\Gamma_0:=\{\beta\gamma-(\beta-1)x\;;\;x\in\{0,1\}\}.\end{array}
\end{equation}
We proceed in the same way for $k=1$:
\begin{equation}\label{secondcase}
\begin{array}{l}\mu_\beta\big(\gamma+\frac1\beta S_1(I)\big)=\frac14\sum_{\gamma_1\in\Gamma_1}\mu_\beta\big(\gamma_1+\frac1\beta I\big),\text{ where}\\
\Gamma_1:=\left\{\beta^2\gamma+1-(\beta-1)(x\beta+y)\;;\;x,y\in\{0,1\}\right\},\end{array}
\end{equation}
and for $k=2$:
\begin{equation}\label{thirdcase}
\begin{array}{l}\mu_\beta\big(\gamma+\frac1\beta S_2(I)\big)=\frac1{16}\sum_{\gamma_2\in\Gamma_2}\mu_\beta\big(\gamma_2+\frac1\beta I\big),\text{ where}\\
\Gamma_2:=\left\{\beta^4\gamma+\beta^2+\beta-(\beta-1)(x\beta^3+y\beta^2+z\beta+t)\;;\;x,y,z,t\in\{0,1\}\right\}.\end{array}
\end{equation}
Since the support of $\mu_\beta$ is included in $[0,1]$ with $\mu_\beta(\{1\})=0$, we can restrict the sums in (\ref{firstcase}), (\ref{secondcase}) and (\ref{thirdcase}) to the indices $\gamma_i$ such that $\big(\gamma_i+\frac1\beta[0,1)\big)\cap[0,1)\ne\emptyset$, that is, $\gamma_i\in(-\frac1\beta,1)$. The relations ${\mathcal R}_0,{\mathcal R}_1,{\mathcal R}_2$ defined by $\gamma{\mathcal R}_i\gamma'\Leftrightarrow\gamma'\in\Gamma_i(\gamma)$ are represented below in the following way: each relation ${\mathcal R}_i$ is represented by the edges with label $i$, and the set of states is the set of the reals that can be reached, from the initial state $i_0=0$, by some path whose states are elements of $(-\frac1\beta,1)$.

\includegraphics[trim=0 0 0 0,scale=0.5]{Figure_3.jpg}

The incidence matrices of the graphs of ${\mathcal R}_0$, ${\mathcal R}_1$ and ${\mathcal R}_2$ are $2M_0,4M_1$ and $16M_2$ respectively, so we deduce from (\ref{firstcase}), (\ref{secondcase}) and (\ref{thirdcase}) that
\begin{equation}\label{translatedcylindersandX}
\left(\begin{smallmatrix}\mu_\beta\big(i_0+\frac1\beta S_{\omega_1}\circ\cdots\circ S_{\omega_n}([0,1))\big)\\\vdots\\\mu_\beta\big(i_6+\frac1\beta S_{\omega_1}\circ\cdots\circ S_{\omega_n}([0,1))\big)\end{smallmatrix}\right)=M_{\omega_1}\cdots M_{\omega_n}x,\text{where }x=\left(\begin{smallmatrix}\mu_\beta\big(i_0+[0,\frac1\beta)\big)\\\vdots\\\mu_\beta\big(i_6+[0,\frac1\beta)\big)\end{smallmatrix}\right).
\end{equation}
It remains to compute $x$. Note that $S_0([0,1))=[0,\frac1\beta)$, $S_1([0,1))=[\frac1\beta,\frac1\beta+\frac1{\beta^2})$ and $S_2([0,1))=[\frac1\beta+\frac1{\beta^2},1)$ form a partition of $[0,1)$. So (\ref{translatedcylindersandX}) (with $n=1$) implies
$$
x=(M_0+M_1+M_2)x.
$$
The sum of the two first entries of $x$ is $\mu_\beta\big([0,\frac2\beta)\big)=\mu_\beta([0,1])=1$, so we can compute the nonnegative eigenvector $x$ and obtain $x=c$.

To prove that $\mu_\beta$ has support $[0,1]$ it is sufficient to find a cover of $[0,1)$ by arbitrarily small intervals of positive measure. For this we choose an arbitrarily large integer $n$ and we consider the intervals $i_0+\frac1{\beta}[\omega_1\dots\omega_n]_{\mathcal S}$ and $i_4+\frac1{\beta}[\omega_1\dots\omega_n]_{\mathcal S}$. They have length less than $\frac1{\beta^n}$ by definition of the maps $S_k$~, and positive measure in consequence of (\ref{translatedcylinders}) because the set of the nonnegative columns-vectors with positive first, third and fifth entries is stable by left-multiplication by $M_k$~, $k\in\{0,1,2\}$. One has $[0,1)=\bigcup_kS_k([0,1))$, hence $[0,1)=\bigcup_{\omega_1\dots\omega_n}[\omega_1\dots\omega_n]_{\mathcal S}$ and consequently
$$\begin{array}{rcl}[0,1)&=&\big[0,\frac1\beta\big)\bigcup\big[1-\frac1\beta,1\big)\\&=&\displaystyle\Big(\bigcup_{\omega_1\dots\omega_n}\big(i_0+\frac1{\beta}[\omega_1\dots\omega_n]_{\mathcal S})\Big)\bigcup\Big(\bigcup_{\omega_1\dots\omega_n}\big(i_4+\frac1{\beta}[\omega_1\dots\omega_n]_{\mathcal S})\Big).\end{array}$$\end{proof}
 
\begin{theorem}\label{propertiesofmu}
The measure $\mu_\beta$ is weak-Gibbs with respect to the system of affine contractions defined in Lemma \ref{SiM} and to the potential defined on $\Omega=\{0,1,2\}^\mathbb N$ by $\Phi(\omega)=\log\Vert M_{\omega_1}c_{\sigma\omega}\Vert$~, where the map $\omega\mapsto c_{\omega}$ is defined in Theorem \ref{weak-Gibbs} and $M_0,M_1,M_2,c$ in Lemma \ref{SiM}. Consequently $\mu_\beta$ satisfies the multifractal formalism.
\end{theorem}

The proof is given in Appendix \ref{example}.

\section{Divergence of the normalized right-products of matrices}\label{limpoints}An immediate consequence of \cite{F}[Theorem 1.1] is that the sequence $n\mapsto\frac{P_n}{\Vert P_n\Vert}$ converges to
a rank one matrix if the matrices $\frac{A_n}{\Vert A_n\Vert}$ are positive and converge to a positive matrix $A$. Theorem \ref{aediv}(ii) below is a partial converse of this result: if the sequence $n\mapsto\frac{A_n}{\Vert A_n\Vert}$ has at least two limit-points $A,A'$ without common left-eigenvector, then $n\mapsto\frac{P_n}{\Vert P_n\Vert}$ diverges. For this we use an argument similar to the one of Elsner \& Friedland in \cite[Theorem~1]{EF97}.

To specify in what sense $n\mapsto\frac{P_n}{\Vert P_n\Vert}=\frac{A_1\cdots A_n}{\Vert A_1\cdots A_n\Vert}$ in general diverges, we use in Theorem \ref{aediv}(iii) the infinite product of a non-atomic Borel probability-measure $p$ with support~$\mathbb C$: the probability that the $(i,j)$-entry of $A_n$ belongs to a Borel set $B_{n,i,j}\subset\mathbb C$ for any $n,i,j$, is $\prod_{n,i,j}p(B_{n,i,j})$.

\begin{theorem}\label{aediv}Let $(A_n)_{n\in\mathbb N}$ be a sequence in the space $\mathcal M_d(\mathbb C)$ of the $d\times d$ matrices on $\mathbb C$, such that $\forall n\ P_n\ne0$.

\noindent(i) If the sequence $n\mapsto\frac{P_n}{\Vert P_n\Vert}$ converges, there exists a row-vector $r\ne0$ and a sequence of positive reals $(\lambda_n)_{n\in\mathbb N}$ such that $\lim_{n\to\infty}r\big(\frac{A_n}{\Vert A_n\Vert}-\lambda_nI_d\big)=0$. 

\noindent(ii) If the sequence $n\mapsto\frac{P_n}{\Vert P_n\Vert}$ converges, the limit-points of the sequence $n\mapsto\frac{A_n}{\Vert A_n\Vert}$ have at least one common left-eigenvector.

\noindent(iii) Let $p$ be a non-atomic Borel probability-measure with support $\mathbb C$ and $p^*$ the product-probability on $\mathcal M_d(\mathbb C)^\mathbb N$. For $p^*$-a.e. sequence $(A_n)_{n\in\mathbb N}\in\mathcal M_d(\mathbb C)^\mathbb N$ the sequence $n\mapsto\frac{P_n}{\Vert P_n\Vert}$ exists and diverges.\end{theorem}

\begin{proof}[{\bf Proof}] (i) The bounded sequence $n\mapsto\lambda_n:=\frac{\Vert P_n\Vert}{\Vert P_{n-1}\Vert\ \Vert A_n\Vert}$ satisfies $\lambda_n\Big(\frac{P_n}{\Vert P_n\Vert}-\frac{P_{n-1}}{\Vert P_{n-1}\Vert}\Big)=\frac{P_{n-1}}{\Vert P_{n-1}\Vert}\Big(\frac{A_n}{\Vert A_n\Vert}-\lambda_nI_d\Big)$ hence, if $n\mapsto\frac{P_n}{\Vert P_n\Vert}$ converges to a matrix $P$, $\lim_{n\to\infty}\frac{P_{n-1}}{\Vert P_{n-1}\Vert}\big(\frac{A_n}{\Vert A_n\Vert}-\lambda_nI_d\big)=0$ and $\lim_{n\to\infty}P\big(\frac{A_n}{\Vert A_n\Vert}-\lambda_nI_d\big)=0$. Since $P$ has norm $1$, it has at least one nonnull row $r$ and $\lim_{n\to\infty}r\Big(\frac{A_n}{\Vert A_n\Vert}-\lambda_nI_d\Big)=0$.

(ii) Using the second triangular inequality, $\lim_{n\to\infty}\big(\big\Vert r\frac{A_n}{\Vert A_n\Vert}\big\Vert-\lambda_n\Vert r\Vert\big)=0$. So for any limit-point $A=\lim_{k\to\infty}\frac{A_{n_k}}{\Vert A_{n_k}\Vert}$, the sequence $k\mapsto\lambda_{n_k}$ converges to $\lambda=\frac{\Vert rA\Vert}{\Vert r\Vert}$ and $rA=\lambda r$.

(iii) In view of (ii) we have to prove that the following sets have probability~$0$:
$$
\begin{array}{rcl}
\mathcal E_0&:=&\{(A_n)_{n\in\mathbb N}\;;\;\exists n,\ P_n=0\},\\
\mathcal E&:=&\big\{(A_n)_{n\in\mathbb N}\;;\;\text{the limit-points of }\frac{A_n}{\Vert A_n\Vert}\text{ have at least one common}\\&&\text{left-eigenvector}\big\}.\end{array}
$$$\mathcal E_0$ is a subset of $\{(A_n)_{n\in\mathbb N}\;;\;\exists n\ \det A_n=0\}$. It is sufficient to prove that the set $\mathcal E_d:=\{M\in\mathcal M_d(\mathbb C)\;;\;\det M=0\}$ has probability $0$. Suppose by induction hypothesis that $\mathcal E_{d-1}$ has probability $0$. Then the co-matrix $M'$ of a matrix $M\in\mathcal M_d(\mathbb C)$ has nonnull $(1,1)$-entry with probability $1$. If this entry is nonnull and $\det M=0$, then $m_{1,1}$ is a function of the other entries of~$M$, $m_{1,1}=f((m_{i',j'})_{(i',j')\ne(i,j)})$. Since $p$ is non-atomic, the conditional probability of $m_{1,1}=f((m_{i',j'})_{(i',j')\ne(i,j)})$ given $(m_{i',j'})_{(i',j')\ne(i,j)}$ is null, hence $\mathcal E_d$ has probability $0$, and $p^*(\mathcal E_0)=0$.

To prove that $p^*(\mathcal E)=0$ we consider the matrices$$T^{(0)}:=\frac2{d(d+1)}\left(\begin{smallmatrix}1&0&\dots&0\\1&1&\dots&0\\\vdots&\vdots&\ddots&\vdots\\1&1&\dots&1\end{smallmatrix}\right)\text{ and }T^{(1)}:=\frac2{d(d+1)}\left(\begin{smallmatrix}1&1&\dots&1\\0&1&\dots&1\\\vdots&\vdots&\ddots&\vdots\\0&0&\dots&1\end{smallmatrix}\right)$$that do not have a common left-eigenvector. We denote by $t^{(0)}_{i,j}$ and $t^{(1)}_{i,j}$ their entries respectively, and by $a^{(n)}_{i,j}$ the ones of $A_n$. Note that $\mathcal E$ is a subset of $\cup_k\cap_{\ell\ge k}\mathcal E_{k,\ell}$ with $\mathcal E_{k,\ell}:=\{(A_n)_{n\in\mathbb N}\;;\;\exists i,j\ \exists h\in\{0,1\}\ \vert a^{(2\ell-h)}_{i,j}-t^{(h)}_{i,j}\vert\ge\frac1k\}$ because, if $(A_n)_{n\in\mathbb N}$ is in the complementary of $\cup_k\cap_{\ell\ge k}\mathcal E_{k,\ell}$, the sequence $n\mapsto\frac{A_n}{\Vert A_n\Vert}$ has limit-points $T^{(0)}$ and $T^{(1)}$ without common left-eigenvector. The set $\{(A_n)_{n\in\mathbb N}\;;\;\vert a^{(2\ell-h)}_{i,j}-t^{(h)}_{i,j}\vert\ge\frac1k\}$ has probability $\le1-\beta$ with $\beta>0$, $\beta$~independent of $h,i,j$, hence $\mathcal E_{k,\ell}$ has probability $\le1-\beta^{2d^2}$ and consequently $p^*(\cap_{\ell\ge k}\mathcal E_{k,\ell})=0$, $p^*(\mathcal E)=0$.\end{proof}

\begin{corollary}\label{nec}Let $(A_n)_{n\in\mathbb N}$ be a non eventually constant sequence of matrices that belong to a finite subset of $\mathcal M_d(\mathbb C)$, say $\mathcal M=\{M_0,\dots,M_{b-1}\}$, and suppose that, for any distinct $i$ and $j$, the matrices $M_i$ and $M_j$ do not have a common left-eigenvector. Then, assuming that the matrix $P_n=A_1\cdots A_n$ is nonnull for any $n$, the sequence $n\mapsto\frac{P_n}{\Vert P_n\Vert}$ diverges.
\end{corollary}

\begin{proof}[{\bf Proof}] The hypotheses imply that the sequence $n\mapsto\frac{A_n}{\Vert A_n\Vert}_n$ is not eventually constant. It admits at least two distinct limit-points $\frac{M_i}{\Vert M_i\Vert}$ and $\frac{M_j}{\Vert M_j\Vert}$ and, by Theorem \ref{aediv}(ii), $n\mapsto\frac{P_n}{\Vert P_n\Vert}$ diverges.
\end{proof}

\begin{remark}If we suppose, in addition of the hypotheses of Corollary \ref{nec}, that the matrices $M_i$ are positive, the sequence $n\mapsto\frac{P_nv}{\Vert P_nv\Vert}$ converges to a limit independent of the positive vector $v$ because the Birkhoff contraction coefficient of $P_n$ tends to $0$, see \cite[\S2]{Har02} or \cite[\S3]{Sen81}.\end{remark}

\begin{remark}Let $(A_n)_{n\in\mathbb N}$ be a non eventually constant sequence of matrices, that belong to the set $\{M_0,M_1,M_2\}$ defined in Lemma \ref{SiM}. Assuming that the sets $\{n\;;\;A_n=M_0\}$ and $\{n\;;\;A_n\in\{M_1,M_2\}\}$ are infinite, we prove that $n\mapsto\frac{P_n}{\Vert P_n\Vert}$ diverges. Indeed $M_1$ has left-eigenvectors $(0,0,0,0,0,a,b)$ for any $(a,b)\in\mathbb C^2\setminus\{(0,0)\}$, and approximately $(1.22,0,1.49,1.82,1,0,0)$ and $(-0.72,0,0.52,-0.38,1,0,0)$, $M_2$ has left-eigenvectors $(a,b,-a,0,-a,c,d)$ for $(a,b,c,d)\in\mathbb C^4\setminus\{(0,0,0,0)\}$, and $(a,0,0,-a,b,0,a)$ for $(a,b)\in\mathbb C^2\setminus\{(0,0)\}$ and, since no of them is eigenvector of $M_0$, $n\mapsto\frac{P_n}{\Vert P_n\Vert}$ diverges by Theorem \ref{aediv}(ii). On the other hand $n\mapsto\frac{P_nv}{\Vert P_nv\Vert}$ converges by Lemma \ref{thnec}.
\end{remark}

\appendix
\section{Proof of Theorem \ref{propertiesofmu}}\label{example}

\subsection{General remarks about Theorem \ref{rankone}(ii)}

To get an idea of the limit-matrices $P^{(s)}$, we need to estimate the ratios $\frac{\Vert P_nu_j\Vert}{\Vert P_nu_{j'}\Vert}$ and the ratios between entries with same column-index. For this we define two coefficients $\Lambda(A),\lambda(A)$ in such a way to obtain a simple formula for $\Lambda(P_n)$ (Lemma~\ref{lL}).

\begin{definition}\label{lambda}Let $\Lambda(A):=\max\frac{\Vert Au_j\Vert}{a_{i,j}}$ ($(i,j)$ such that $a_{i,j}\ne0$) and $\lambda(A):=\max\frac{\Vert Au_k\Vert}{a_{i,j}}$ ($(i,j,k)$ such that $a_{i,j}\ne a_{i,k}=0$), or $0$ if $\not\exists (i,j,k)$.\end{definition}\begin{remark}\label{normofproduct}For any $A$ and $B$ nonnegative, $\Vert AB\Vert\ge\Vert A\Vert\Vert B\Vert/\Lambda(B)$ because $\Vert A\Vert\Vert B\Vert=\sum_{i,j,j',k}a_{i,j}b_{j',k}$ and $\sum_{j'}b_{j',k}\le\Lambda(B)b_{j,k}$ if $b_{j,k}\ne0$.\end{remark}\begin{lemma}\label{lL}$\forall A_1,\dots,A_n$ nonnegative, $\displaystyle\Lambda(P_n)\le\sum_{k=1}^n\Lambda(A_k)\prod_{1\le i<k}\lambda(A_i)$.\end{lemma}

\begin{proof}[{\bf Proof}]It is sufficient to prove this inequality for $P=AB$, product of two nonnegative matrices. Suppose $p_{i_0,j_0}\ne0$ or, equivalently, $\exists j_1\ a_{i_0,j_1}b_{j_1,j_0}\ne0$.$$\begin{array}{rcl}
\Vert Pu_{j_0}\Vert&=&\sum\nolimits_{a_{i_0,j}\ne0}\Vert Au_j\Vert b_{j,j_0}+
\sum_{a_{i_0,j}=0}\Vert Au_j\Vert b_{j,j_0}\\
&\le&\Lambda(A)\sum_ja_{i_0,j}b_{j,j_0}+
\lambda(A)a_{i_0,j_1}\Vert Bu_{j_0}\Vert\\
&\le&\Lambda(A)p_{i_0,j_0}+\lambda(A)a_{i_0,j_1}\Lambda(B)b_{j_1,j_0}\\
&\le&\Lambda(A)p_{i_0,j_0}+\lambda(M)\Lambda(N)p_{i_0,j_0},\end{array}$$proving that $\Lambda(P)\le\Lambda(A)+\lambda(A)\Lambda(B)$ (or $\Lambda(P)\le\Lambda(A)$ if $\lambda(A)=0$).\end{proof}

\begin{remark}$\Lambda(A)$ is related to $\Lambda'(A):=\max_{i,i',j\atop a_{i,j}\ne0}\frac{a_{i',j}}{a_{i,j}}$ by the inequalities $\Lambda'(A)\le\Lambda(A)\le\Lambda'(A)\cdot d$, and $\Lambda'(A)$ is related to the Hausdorff distance $d_H$ defined in \cite[\S3]{Sen81}: $\log\Lambda'(A)=\max_j d_H(Au_j,u)$. Another formula for $P_n$ is $\max_{j,j'}d_H(P_nu_j,P_nu_{j'})\le\tau_B(A_1)\cdots\tau_B(A_{n-1})\max_{j,j'}d_H(A_nu_j,A_nu_{j'})$, where $\tau_B(A)\le1$ for any nonnegative matrix $A$, $\tau_B(A)<1$ if $A$ is positive.\end{remark}

\begin{proposition}\label{cd1}Any sequence of nonnegative $d\times d$ matrices $(A_n)_{n\in\mathbb N}$ and any limit-point $\lim_{k\to\infty}\frac{P_{n_k}}{\Vert P_{n_k}\Vert}$ such that $\displaystyle\sup_{k\ge1,\ n_{k+1}\le n<n_{k+2}}\Lambda(P_{n_k,n})<\infty$ satisfy the first condition of Theorem \ref{rankone}(ii).\end{proposition}\begin{proof}[{\bf Proof}]We use the inequality of Remark \ref{normofproduct} with $A=P^{(s)}$ and $B=P_{n_k,n}$, $n_{k+1}\le n<n_{k+2}$.\end{proof}

Since the uniform convergence is required in Theorem \ref{weak-Gibbs}, we will need the following general result habitually used for similar problems.\begin{proposition}\label{pointwisuniform}Given a set of matrices $\mathcal M=\{M_0,\dots,M_{b-1}\}$, a column-vector $c$ and $\Omega:=\{0,\dots,b-1\}^\mathbb N$, the sequence of column-vectors $n\mapsto c_{\omega,n}=\frac{M_{\omega_1}\cdots M_{\omega_n}c}{\Vert M_{\omega_1}\cdots M_{\omega_n}c\Vert}$ converges uniformly over $\Omega$ if and only if the following pointwise convergence holds for any $\omega\in\Omega$\begin{equation}\label{modifunif}\lim_{n\to\infty}\big(\sup\{\Vert c_{\xi,r}-c_{\xi,s}\Vert\;;\;\xi\in[\omega_1\dots\omega_n],r\ge n,s\ge n\}\big)=0.\end{equation}\end{proposition}\begin{proof}[{\bf Proof}]The direct implication is given by the Cauchy criterion. Conversely, suppose that (\ref{modifunif})~holds for any $\omega\in\Omega$. Given $\varepsilon>0$ and $\omega\in\Omega$ there exists a rank $n(\omega)$ such that$$\xi\in[\omega_1\dots\omega_{n(\omega)}],r\ge n(\omega),s\ge n(\omega)\ \Rightarrow\ \left\Vert c_{\xi,r}-c_{\xi,s}\right\Vert\le\varepsilon.$$Each cylinder set $[\omega_1\dots\omega_{n(\omega)}]$ is an open set containing $\omega$. Because $\Omega$ is compact, it is covered by finitely many of such sets, say $[\omega^{i}_1\dots\omega^{i}_{n(\omega^i)}]$ for $i=1,\dots,N$. Let $\displaystyle n=\max_{i\in\{1,\dots,N\}}n(\omega^i)$. The inequality $\left\Vert c_{\xi,r}-c_{\xi,s}\right\Vert\le\varepsilon$ holds for any $\omega\in\Omega,\xi\in[\omega_1\dots\omega_n],r\ge n,s\ge n$, proving that the sequence $n\mapsto c_{\omega,n}$ is uniformly Cauchy and converges uniformly on $\Omega$.\end{proof}

\subsection{Proof of Theorem \ref{propertiesofmu} by means of Theorem \ref{weak-Gibbs}}

The probability measure, defined on $\{0,1,2\}^\mathbb N$ by$$
\nu_\beta([\omega_1\dots\omega_n]):=\mu_\beta([\omega_1\dots\omega_n]_\mathcal S),
$$
is sofic because one has$$[\omega_1\dots\omega_n]_\mathcal S=\left\{\begin{array}{ll}i_0+\frac1\beta[\omega_2\dots\omega_n]_\mathcal S&\text{if }\omega_1=0\\i_1+\frac1\beta[0\omega_2\dots\omega_n]_\mathcal S&\text{if }\omega_1=1\\i_1+\frac1\beta[10\omega_2\dots\omega_n]_\mathcal S&\text{if }\omega_1=2\end{array}\right.$$and, with Lemma \ref{SiM},$$\nu_\beta\big([\omega_1\dots\omega_n]\big)=\left\{\begin{array}{ll}{^tu_1}M_{\omega_2}\cdots M_{\omega_n}c&\text{if }\omega_1=0\\\frac12\ {^tu_3}M_{\omega_2}\cdots M_{\omega_n}c&\text{if }\omega_1=1\\\frac18\ {^tu_5}M_{\omega_2}\cdots M_{\omega_n}c&\text{if }\omega_1=2.\end{array}\right.$$So it is sufficient to to check that $u_1,u_3,u_5,c$ and the $M_k$ satisfy (\ref{modifunif}), (\ref{randc}),~(\ref{Mandc}).\begin{remark}By the second part of the proof of \cite[Theorem 7]{T2}, $\nu_\beta$ is the image by the morphism $(\epsilon_n)_{n\in\mathbb N}\mapsto(\lfloor\frac{\epsilon_n}7\rfloor)_{n\in\mathbb N}$ of the Markov probability measure on $\{0,\dots,20\}^\mathbb N$ with initial probabilities

$\left(\begin{smallmatrix}\frac35&0&0&0&0&0&0&0&0&\frac{13}{40}&0&0&0&0&0&0&0&0&\frac3{40}&0&0\end{smallmatrix}\right)$ and transition matrix$$P:=\left(\begin{smallmatrix}M\\M\\M\end{smallmatrix}\right)\text{ with }M:=\left(\begin{smallmatrix}\frac12&0&0&0&0&0&0&0&0&\frac{13}{48}&\frac1{12}&0&0&0&\frac1{16}&0&0&0&\frac1{16}&0&\frac1{48}\\0&0&\frac{13}{16}&0&0&0&0&0&0&0&0&0&\frac3{16}&0&0&0&0&0&0&0&0\\0&0&0&\frac2{13}&\frac6{13}&0&0&0&0&0&\frac1{13}&\frac3{13}&0&0&\frac3{52}&0&0&0&0&0&\frac1{52}\\0&0&0&0&0&0&0&\frac34&0&0&0&0&0&0&0&0&0&\frac1{16}&\frac3{16}&0&0\\\frac12&0&0&0&0&0&\frac16&0&0&\frac{13}{48}&0&0&0&0&0&0&0&0&\frac1{16}&0&0\\0&0&0&0&1&0&0&0&0&0&0&0&0&0&0&0&0&0&0&0&0\\0&1&0&0&0&0&0&0&0&0&0&0&0&0&0&0&0&0&0&0&0\end{smallmatrix}\right).$$\end{remark}

We describe $\mathcal I(M_{\omega_1}\cdots M_{\omega_n})$ in the following Lemma (note that in much cases this set does not have the form $I\times J$, for instance when $\omega_1\omega_2\dots=(1100000)^{i_1}(1012)^{i_2}(1100000)^{i_3}(1012)^{i_4}\dots$ with $i_1,i_2,\dots\ge1$). Let

$A:=2M_0$, $B:=4M_1$, $C:=16M_2$ (matrices with integral entries)

$\mathcal C_1:=\{v\;;\;v\ge u_{135}\}$, that is, $\{v\;;\;v-(u_1+u_3+u_5)\text{ nonnegative}\}$

$\mathcal C_2:=\{u_j\ {\scriptstyle(j\le7)},u_{125},u_{13},u_{5},u_{23},u_{236},u_{27},u_{34},u_{356},u_{36},u_{57},u_{145},u_{134},u_{1334}\}$ $\mathcal E:=\{ABA,B^2A,CBA,CA\}$

$\mathcal E':=\{A^2C,A^2BAC,BABAC,CABAC,B^2AC,CBAC,CAC,BC\}$

$\mathcal S=\{EA^\alpha B^\beta\;;\;E\in\mathcal E,\ \alpha\ge0,\ \beta\ge0\}\cup\{E'C^\alpha B^\beta\;;\;E'\in\mathcal E',\ \alpha\ge0,\ \beta\ge0\}$.

We use the stability of $\mathcal C_1$ and $\mathcal C_1\cup\mathcal C_2$ by left-multiplication by $A$, $B$ or $C$. We eventually consider the elements of $\mathcal S$ as words on the alphabet $\{A,B,C\}$.

\begin{lemma}\label{mathcalI}(i) For any sequence $(A_n)_{n\in\mathbb N}\in\{A,B,C\}^\mathbb N$ there exist

$s_0:=0\le  s_1<s_2<\dots$ such that $S_k:=P_{s_k,s_{k+1}}\in\mathcal S$, except $S_0$, empty or not empty, strict suffix of an element of $\mathcal S$. For any $n$,
$$P_n=S_0\cdots S_{k(n)}T_n,\ T_n\text{ strict prefix of }S_{k(n)+1},$$specifying that, if $(A_n)_{n\in\mathbb N}$ is eventually constant, the sequence $s_0<s_1<\dots<s_\ell$ is finite and, for $n$ large enough, $P_n=S_0\cdots S_\ell T_n$ with 
$T_n\in\mathcal S$.

(ii) For any $S_1,\dots,S_k\in\mathcal S$ with $k\ge3$, $\mathcal I(S_1\cdots S_k)=(I\times J)\cup(I'\times J')$ with $\{1,3,5\}\cup I'\subset I$ and $J\cap\{1,3,5\}\ne\emptyset$.

(iii) There exist $\Lambda<\infty$, upper bound of $\Lambda(S_1\cdots S_{k+1})$ for $S_1,\dots,S_k\in\mathcal S$ and $S_{k+1}$ prefix empty or not empty of an element of $\mathcal S$.\end{lemma}




\begin{proof}[{\bf Proof}](i) For $n$ large enough there exists $m$ such that $A_m\dots A_n\in\mathcal S$. So $A_1\dots A_n=S_0(n)S_1(n)\dots S_{k(n)}(n)$ where $S_k(n)\in\mathcal S$ for $k\ge1$ and the word $S_0(n)$ is a strict suffix of a word of $\mathcal S$. The word $S_0(n)$ takes the same value $S_0$ for $n\in I_0$ (infinite set). There exists an infinite set $I_1\subset I_0$ such that the couple $(S_0(n),S_1(n))$ takes the same value $(S_0,S_1)$ for $n\in I_1$. And so one, finally $A_1A_2\dots=S_0S_1S_2\dots$.

(ii) We compute the matrices of $\mathcal E\cup\mathcal E'$ and we see that their first or third or fifth column belongs to $\mathcal C_1$. When two distinct columns belong to $\mathcal C_2$, they are $u_2$ and $u_{23}$, or $u_{15}$ and $u_{34}$, or $u_{13}$ and $u_{145}$. Both properties remain true for the matrices $S\in\mathcal S$, because the stability of $\mathcal C_1$ by left-multiplication by any matrix $M\in\{A,B,C\}$ implies \hbox{$\forall j\in\{1,3,5\}\ \exists j'\in\{1,3,5\}\ Mu_{j'}=u_j$}, and because the columns of $S$ are nonnegative linear combinations of the columns of one matrix of $\mathcal E\cup\mathcal E'$.

To ensure that $S_{k-1}S_k$ does not have two distinct columns that belong to~$\mathcal C_2$, it is sufficient to check that the following couples of vectors are equal when both belong to $\mathcal C_2$: $S_{k-1}u_2=S_{k-1}u_{23}$ and $S_{k-1}u_{15}=S_{k-1}u_{34}$ and $S_{k-1}u_{13}=S_{k-1}u_{145}$. For instance suppose $S_{k-1}=ABAA^\alpha B^\beta$; then we have $S_{k-1}u_2=0\not\in\mathcal C_2$ in both cases $\beta\ge1$ and $\beta=0$ with $\alpha\in\{0,3\}$ mod.$4$; we have $S_{k-1}u_2=S_{k-1}u_{23}$ in the case $\beta=0$ with $\alpha\equiv1$ mod.$4$; we have $S_{k-1}u_2\in\mathcal C_1$ in the case $\beta=0$ with $\alpha\equiv2$ mod.$4$. So there exist two sets $J$ (containing $1$ or $3$ or $5$) and $J'$ such that $S_{k-1}S_ku_j\in\mathcal C_1$ for $j\in J$, and $S_{k-1}S_ku_{j'}\in\mathcal C_2$ is independent of \hbox{$j'\in J'$}, as well as $Pu_{j'}$. 

Now we use the following ''synchronizing property'': if $M\in\{A^2,BA,B^2,C\}$ and $(u,v,w)\in\mathcal C_1\times \mathcal C_2\times \mathcal C_2$, then $I(Mu)\subset\mathcal I(Mv)=\mathcal I(Mw)$. The elements of~$\mathcal E\cup\mathcal E'$ have at least one factor $A^2$, $BA$, $B^2$ or~$C$ hence they have this property. Consequently if $k\ge3$ the set $\mathcal I(Pu_j)$ does not depend on $j\in J$.

(iii) We use the formulas $A^{4n}=\left(\begin{smallmatrix}1&0&0&0&0&0&0\\n&0&1&0&0&0&0\\n&0&0&1&1&0&0\\0&0&0&0&0&0&0\\n&0&0&0&0&0&1\\n&0&0&0&1&0&0\\n&1&0&0&0&0&0\end{smallmatrix}\right)$ and $C^n=\left(\begin{smallmatrix}1&0&0&0&n&0&1\\0&0&0&0&0&0&0\\1&0&0&0&n-1&0&1\\0&0&0&1&n&0&0\\0&0&0&0&1&0&0\\0&0&0&0&0&0&0\\0&0&0&0&0&0&0\end{smallmatrix}\right)$. The matrices of $\mathcal S$ are products of at most three matrices of the set $\mathcal X:=\{A,A^2,A^3,EA^{4n},B^n,E'C^n\}_{E\in\mathcal E,E'\in\mathcal E',n\ge0}$. Since $\max_{X\in\mathcal X}\Lambda(X)=12$ and $\max_{X\in\mathcal X}\lambda(X)=\frac52$, Lemma \ref{lL} gives $\Lambda(S_1\cdots S_k)\le\sum_{i=0}^{3k-1}12\cdot(\frac52)^i=:C_k$.

To obtain a better bound we first prove that $\lambda(S_1\cdots S_9)<1$. We have $S_1=E_1X^\alpha B^\beta$ with $(E_1,X)\in\mathcal E\times\{A\}$ or $\mathcal E'\times\{C\}$ and, for $j\in J'$, either the vector $A^\alpha B^\beta S_2\cdots S_9u_{j'}$ belongs to~$\mathcal C_1$ and in this case the ''synchronizing property'' implies that $\mathcal I(S_1\cdots S_9u_j)$ does not depend on $j\in J\cup J'$, and $\lambda(S_1\cdots S_9)=0$, either it belongs to $\mathcal C_2$ and $S_1\cdots S_9u_{j'}$, image of a vector of $\mathcal C_2$ by a matrix of $\mathcal E\cup\mathcal E'$, has norm at most $7$.

Let now $j\in J$. By definition of $J$, $S_8S_9u_j\ge u_{135}$. Let us find a lower bound for $S_6S_7u_{135}$. We write $S_6=E_6Y^{\alpha'}B^{\beta'}$ and $S_7=E_7Z^{\alpha''}B^{\beta''}$, $(E_6,Y)$ and $(E_7,Z)\in\mathcal E\times\{A\}$ or $\mathcal E'\times\{C\}$. Let $v:=E_7u_{135}$, we have $S_6S_7u_{135}\ge S_6v$. If $E_7=ABA$ then $v=u_{1233567}$; if $E_7=A^2C$ or $A^2BAC$ then $v\ge2u_{135}$; and in the other cases $v\ge u_{11345}$.

We want to prove the inequality $S_6v_k\ge2u_{135}$. The images of $u_{1233567}$ and $u_{11345}$ by $B$ or $C$ being $\ge u_{11345}$, it remains to check that $E_6A^{\alpha'}u_{1233567}\ge2u_{135}$ and $E_6A^{\alpha'}u_{11345}\ge2u_{135}$, and only for $\alpha'\in\{0,1,2,3\}$ because \hbox{$A^5\ge A$}. We see that both inequalities are true except in the case $E_6\in\{ABA,CBA\}$ with $\alpha'=0$ and $v=u_{1233567}$. This value of $v$ corresponds to $\beta'=0$ and $E_7=ABA$, hence in this case $BA^2BA$ is a factor of $S_6S_7$, in which case $\lambda(S_1\cdots S_9)=0$. Assuming that $\lambda(S_1\cdots S_9)\ne0$, we have proved that $S_6S_7u_{135}\ge2u_{135}$ and similarly $S_4S_5u_{135}\ge2u_{135}$, $S_2S_3u_{135}\ge2u_{135}$, so $S_2\cdots S_7u_{135}\ge8u_{135}$. According to the definition of $J$, $S_2\cdots S_9u_j\ge2^{\lfloor\frac {k-3}2\rfloor}u_{135}$ for any $j\in J$. The image of the l.h.s. by $S_1$ is $Pu_j$, and the image of the r.h.s. by $S_1$ has entries either $0$ or at least $8$ but, by the ''synchronizing property'', the location of the nonnull entries in both vectors is the same, so the nonnull entries of $S_1\cdots S_9u_j$ are at least $8$.

By Lemma \ref{lL}, $\Lambda(S_1\cdots S_k)=\Lambda((S_1\cdots S_9)(S_{10}\cdots S_{18})\cdots)\le C_9\sum_0^\infty(\frac78)^k$. This is also true if $k<9$. Using still Lemma \ref{lL}, $\Lambda(S_1\cdots S_{k+1})$ is bounded because if $S_{k+1}\not\in\mathcal S$, it is the product of at most four matrices $A$, $B$ or~$C$.\end{proof}

\begin{lemma}\label{thnec}If the sequence of nonnegative matrices $\mathcal A=(A_n)_{n\in\mathbb N}$ is not eventually constant, the sequence $n\mapsto\frac{P_nv}{\Vert P_nv\Vert}$ converges uniformly to a constant $v_\mathcal A$ on the set $\mathcal V_\alpha:=\{v\;;\;\Vert v\Vert=1,v_1\ge\alpha,v_3\ge\alpha,v_5\ge\alpha\}$ ($0<\alpha\le\frac13$).\end{lemma}

\begin{proof}[{\bf Proof}]Let $(s_k)_{k\in\mathbb N}$ be the sequence involved in Lemma \ref{mathcalI}(i), and $(n_k)_{k\in\mathbb N}$ a subsequence of $(s_k)_{k\in\mathbb N}$ such that $k\mapsto\frac{P_{n_k}}{\Vert P_{n_k}\Vert}$ converges. To apply Proposition \ref{cd1} we must check that \hbox{$\displaystyle\sup_{k\ge1,\ n_{k+1}\le n<n_{k+2}}\Lambda(P_{n_k,n})<\infty$}. This follows from Lemma \ref{mathcalI}(iii) because there exists $h$ such that $n_k=s_h$. So the first condition of Theorem \ref{rankone}(ii) is satisfied, let us check the second. The matrix $P_{n_k,n_{k+1}}$ is a product $S_h\cdots S_{h'}$ and we can chose the sequence $(n_k)_{k\in\mathbb N}$ for $h'-h$ to be at least $2$, consequently Lemma \ref{mathcalI}(ii) applies to $P_{n_k,n_{k+1}}$ and $\mathcal I(P_{n_k,n_{k+1}})$ has the form $(I_k\times J_k)\cup(I'_k\times J'_k)$ with $I'_k\subset I_k$. By Lemma \ref{mathcalI}(iii) the ratio between any couple of entries of $P_{n_k,n_{k+1}}$ with same column-index has a positive lower bound, hence the set $\mathcal I(Q)$ (when $Q$ is a limit-point of $k\mapsto\frac{P_{n_k,n_{k+1}}}{\Vert P_{n_k,n_{k+1}}\Vert}$) has the form $(I\times J)\cup(I'\times J')$ with $I=I_k,I'=I'_k,J\subset J_k,J'\subset J'_k$ for $k$ large enough. Since $I'\subset I$, the second condition of Theorem \ref{rankone}(ii) is satisfied.

Let $v=\alpha u_{135}+w$ be an element of $\mathcal V_\alpha$. It has coordinates $\frac{w_i}{1-3\alpha}$ in the basis $e_i:=\alpha u_{135}+(1-3\alpha)u_i$. Since $\frac{P_ne_i}{\Vert P_ne_i\Vert}$ converges to a vector $v_\mathcal A$, one has $\Vert P_ne_i-\Vert P_ne_i\Vert v_\mathcal A\Vert\le\varepsilon_n\Vert P_ne_i\Vert$ with $\lim_{n\to\infty}\varepsilon_n=0$. By the triangular inequality, $\Vert P_nv-\Vert P_nv\Vert v_\mathcal A\Vert\le\varepsilon_n\Vert P_nv\Vert$ and $\frac{P_nv}{\Vert P_nv\Vert}$ converges uniformly.\end{proof}

\begin{proof}[{\bf End of the proof of Theorem \ref{propertiesofmu}}]We prove that (\ref{modifunif}) holds for the matrices $M_0,M_1,M_2$ and the vector $c$ defined in Lemma \ref{SiM}. Given $\varepsilon>0$, let us find an integer $n_\varepsilon$ for $\Vert c_{\xi,r}-c_{\xi,s}\Vert\le\varepsilon$ to hold when $\xi\in[\omega_1\dots\omega_{n_\varepsilon}]$ and $r,s\ge n_\varepsilon$. In the sequel $A_n$ is the matrix $2^{2^{\xi_n}}M_{\xi_n}$, that belongs to $\{A,B,C\}$.

We suppose first that $\omega$ is not eventually constant and we apply Lemma~\ref{thnec}:\begin{equation}\label{unifomega}\textstyle\exists n'_\varepsilon\ \forall n\ge n'_\varepsilon\ \forall v\in\mathcal V_{\frac1\Lambda}\ \big\Vert\frac{P_nv}{\Vert P_nv\Vert}-v_\mathcal A\big\Vert\le\frac\varepsilon2.\end{equation}Choosing $n_\varepsilon$ large enough, there exist four occurrences of $0^h$ or $2^h$ in the word $\omega_{n'_\varepsilon+1}\dots\omega_{n_\varepsilon}$ seen as a concatenation of words $0^h$, $1^h$ and $2^h$ and consequently, by definition of the sequence $(s_k)_{k\in\mathbb N}$ involved in Lemma \ref{mathcalI}(i), there exists $k$ such that $n'_\varepsilon\le s_k\le n_\varepsilon<s_{k+1}$. For any $r\ge n_\varepsilon$, $\Lambda(P_{s_k,r})\le\Lambda$ by Lemma~\ref{mathcalI}(iii), and the inequality $\Vert c_{\xi,r}-c_{\xi,s}\Vert\le\varepsilon$ follows because, if we replace in (\ref{unifomega}) the integer $n$ by $s_k$ and $v$ by $\frac{P_{s_k,r}c}{\Vert P_{s_k,r}c\Vert}$ (element of $\mathcal V_{\frac1\Lambda}$), we obtain $\forall r\ge n_\varepsilon\ \Vert c_{\xi,r}-v_\mathcal A\Vert\le\frac\varepsilon2$.

Suppose now that $\omega$ is eventually constant. In other words, there exists $M\in\{A,B,C\}$ and $n_\omega$ such that $P_rc=P_{n_\omega}M^{n_\varepsilon-n_\omega}P_{n_\varepsilon,r}c$ if $n_\varepsilon$ is large enough.

We use again Lemma \ref{mathcalI}(i) and we have$$P_rc=P_{n_\omega}M^nP_{s_k,r}c$$where $k$ is the smallest integer such that $s_k\ge n_\varepsilon-3$, and $n=s_k-n_\omega$.

By the classical formula about the powers of $M$, there exists $(R_n)_{n\in\mathbb N}$ with limit $0$ such that $M^n=\varphi(n)(vw+R_n)$, where $\varphi(n)$ has the form $n^\alpha\rho^n$ and $v$~(resp. $w$) is a suitable right-eigenvector (resp. left-eigenvector) of $M$. Using the usual inequality $\big\Vert\frac x{\Vert x\Vert}-\frac y{\Vert y\Vert}\big\Vert\le\frac{2\Vert x-y\Vert}{\max(\Vert x\Vert,\Vert y\Vert)}$ we deduce$$\textstyle\big\Vert\frac{P_rc}{\Vert P_rc\Vert}-\frac{P_{n_\omega}\varphi(n)vwP_{s_k,r}c}{\Vert P_{n_\omega}\varphi(n)vwP_{s_k,r}c\Vert}\big\Vert\le2\frac{\Vert P_{n_\omega}\varphi(n)R_n P_{s_k,r}c\Vert}{\Vert P_{n_\omega}\varphi(n)vwP_{s_k,r}c\Vert}.$$After simplification this is equivalent to$$\textstyle\big\Vert\frac{P_rc}{\Vert P_rc\Vert}-\frac{P_{n_\omega}v}{\Vert P_{n_\omega}v\Vert}\big\Vert\le2\frac{\Vert P_{n_\omega}R_n P_{s_k,r}c\Vert}{\Vert P_{n_\omega}v\Vert\Vert wP_{s_k,r}c\Vert}.$$We bound the r.h.s. by $2\frac{\Vert P_{n_\omega}\Vert}{\Vert P_{n_\omega}v\Vert}\Vert R_n\Vert\frac{\Vert P_{s_k,r}c\Vert}{\Vert wP_{s_k,r}c\Vert}$, where $\Vert R_n\Vert$ tends to $0$ and, by Remark \ref{normofproduct}, $\Vert wP_{s_k,r}c\Vert\ge\Vert w\Vert\Vert P_{s_k,r}c\Vert/\Lambda$. The inequalities

$\Vert c_{\xi,r}-\frac{P_{n_\omega}v}{\Vert P_{n_\omega}v\Vert}\Vert\le\frac\varepsilon2$ ($r\ge n_\varepsilon$) and $\Vert c_{\xi,r}-c_{\xi,s}\Vert\le\varepsilon$ ($r,s\ge n_\varepsilon$)
follow if we chose $n_\varepsilon$ large enough.

To check (\ref{randc}) we bound $\Lambda(c_{\omega,n})=\Lambda(P_nc)$. We use Lemma \ref{mathcalI} and apply Lemma \ref{lL}: since $S_0$ is the product of at most three matrices of the set $\{A,B,C\}$ by a power of $A$ or $C$, and by a power of $B$, since $\Lambda(M^n)$ and $\lambda(M^n)$ are in $O(n)$ for $M=A$, $B$ or $C$, $\Lambda(P_nc)=O(n)$ and (\ref{randc}) holds.

(\ref{Mandc}) also holds because $\mathcal I(c_\omega)$ contains $\{1,3,5\}$ or, if $\forall n\ \omega_n=0$ (resp. if $\forall n\ \omega_n=2$), $\mathcal I(c_\omega)=\{2,3,5,6,7\}$ (resp. $\{1,3,4\}$).\end{proof}

\section{Upper triangular $3\times3$ matrices}\label{3}Let $A_n=\left(\begin{smallmatrix}a_n^{1,1}&a_n^{1,2}&a_n^{1,3}\\0&a_n^{2,2}&a_n^{2,3}\\0&0&a_n^{3,3}\end{smallmatrix}\right)$, with $a_n^{i,j}>0$, belong to a finite set; this imply that the ratios $\frac{a_n^{i,j}}{a_n^{i',j'}}$ are bounded. Because of the left-eigenvector ${^tu_1}$ Theorem \ref{aediv} does not apply, and Theorem~\ref{rankone}(ii) is not usefull because one can compute$$P_n=\left(\begin{smallmatrix}\alpha^{1,1}_n&\alpha^{2,2}_ns_n^{1,2}&\alpha^{3,3}_n(s_n^{1,3}+t_n)\\0&\alpha^{2,2}_n&\alpha^{3,3}_ns^{2,3}_n\\0&0&\alpha^{3,3}_n\end{smallmatrix}\right)\text{ with }\left\{\begin{array}{l}\alpha^{i,i}_n:=a_1^{i,i}\cdots a_n^{i,i}\\s^{i,j}_n:=\sum_{k=1}^n\alpha^{i,i}_{k-1}a^{i,j}_k/\alpha^{j,j}_k\\t_n:=\sum_{k=2}^ns_{k-1}^{1,2}\alpha^{2,2}_{k-1}a^{2,3}_k/\alpha^{3,3}_k.\end{array}\right.$$To prove the existence of $\lim_{n\to\infty}\frac{P_n}{\Vert P_n\Vert}$ and $\lim_{n\to\infty}\frac{P_nv}{\Vert P_nv\Vert}$ ($v$ positive) we look at the ratios between the entries of $P_n$. We use the following general~fact: any sequence of the form $n\mapsto\frac{\sum_1^nu_kv_k}{\sum_1^nv_k}$  with $u_k$ positive non-decreasing, $v_k$ positive, is itself non-decreasing and, if $\sum_1^\infty v_k=\infty$, it has the same limit as $n\mapsto u_n$. So if \hbox{$\sum_{n\ge1}\frac{\alpha^{2,2}_n}{\alpha^{3,3}_n}=\infty$} (or equivalently if $\lim_{n\to\infty}s_n^{2,3}=\infty$), $\frac{t_n}{s_n^{2,3}}$ has the same limit as $s_n^{1,2}$.

$1^{\rm st}$ case: $\sum_{n\ge1}\frac{\alpha^{1,1}_n}{\alpha^{2,2}_n}$ and $\sum_{n\ge1}\frac{\alpha^{2,2}_n}{\alpha^{3,3}_n}$ are infinite, in this case $\lim_{n\to\infty}\frac{P_n}{\Vert P_n\Vert}$ has the form $\left(\begin{smallmatrix}a&b&c\\0&0&0\\0&0&0\end{smallmatrix}\right)$ and $\lim_{n\to\infty}\frac{P_nv}{\Vert P_nv\Vert}=\left(\begin{smallmatrix}1\\0\\0\end{smallmatrix}\right)$.

$2^{\rm nd}$ case: $\sum_{n\ge1}\frac{\alpha^{1,1}_n}{\alpha^{2,2}_n}<\infty=\sum_{n\ge1}\frac{\alpha^{2,2}_n}{\alpha^{3,3}_n}$, then $\lim_{n\to\infty}\frac{\alpha^{1,1}_n}{\alpha^{2,2}_n}=0$. Since $\lim_{n\to\infty}s_n^{2,3}=\infty$, this implies $\lim_{n\to\infty}\frac{s^{1,3}_n}{s^{2,3}_n}=0$ and $\lim_{n\to\infty}\frac{P_n}{\Vert P_n\Vert}$ has the form $\left(\begin{smallmatrix}0&a&b\\0&c&d\\0&0&0\end{smallmatrix}\right)$, $ad-bc=0$, and $\exists x,y\ \forall v>0\ \lim_{n\to\infty}\frac{P_nv}{\Vert P_nv\Vert}=\left(\begin{smallmatrix}x\\y\\0\end{smallmatrix}\right)$.

$3^{\rm rd}$ case: $\sum_{n\ge1}\frac{\alpha^{1,1}_n}{\alpha^{2,2}_n}=\infty>\sum_{n\ge1}\frac{\alpha^{2,2}_n}{\alpha^{3,3}_n}$. In this case $\lim_{n\to\infty}\frac{P_n}{\Vert P_n\Vert}$ can have the form $\left(\begin{smallmatrix}a&b&c\\0&0&0\\0&0&0\end{smallmatrix}\right)$, or $\left(\begin{smallmatrix}a&b&c\\0&0&d\\0&0&e\end{smallmatrix}\right)$ and in this case $\lim_{n\to\infty}\frac{P_nv}{\Vert P_nv\Vert}$ depends on $v$.

$4^{\rm th}$ case: $\sum_{n\ge1}\frac{\alpha^{1,1}_n}{\alpha^{2,2}_n}$ and $\sum_{n\ge1}\frac{\alpha^{2,2}_n}{\alpha^{3,3}_n}$ are finite. In this case $\lim_{n\to\infty}\frac{P_n}{\Vert P_n\Vert}$ has the form $\left(\begin{smallmatrix}a&b&c\\0&d&e\\0&0&f\end{smallmatrix}\right)$ with $abcdef\ne0$, and $\lim_{n\to\infty}\frac{P_nv}{\Vert P_nv\Vert}$ depends on $v$.

\vfill
\section{The triangular form}\label{tptc}

\begin{definition}Let

$$\begin{array}{lcl}
\mathcal N(A)&:=&\text{ the number of distinct column of $\mathcal Z(A)$}
\\
\mathcal T_\mathcal J&:=&\text{the set of the upper-block-triangular $d\times d$ matrices}\\&&\text{with respect to the partition $\mathcal J=\{J_1,\dots,J_\delta\}$ of $\{1,\dots,d\}$},\\&&(\text{this means that }\mathcal I(A)\subset\cup_k((J_1\cup\dots\cup J_k)\times J_k))\\
\mathcal T_\mathcal J^1&:=&\text{the set of the matrices }A\in\mathcal T_\mathcal J\text{ such that }\mathcal N(B_{i',j'})=1\\&&\text{for any block }B_{i',j'}:=(a_{i,j})_{(i,j)\in J_{i'}\times J_{j'}}.\end{array}$$
\end{definition}


\begin{theorem}\label{triangular}Let $\mathcal A=(A_n)_{n\ge1}$ be a sequence of nonnegative $d\times d$ matrices. There exist an increasing sequence of nonnegative integers $(n_k)_{k\ge1}$ and a partition $\mathcal J=\mathcal J_\mathcal A$ (unique if we choose $(n_k)_{k\ge1}$ minimal for the lexicographic order) such that, for any $1\le k<\ell$, $P_{n_k,n_\ell}\in T^1_\mathcal J$ and $\mathcal Z(P_{n_k,n_\ell})$ does not depend on $(k,\ell)$.

\end{theorem}

\begin{proof}[{\bf Proof}]Let $\delta=\delta_\mathcal A:=\lim_{n\to\infty}\limsup_{n'\to\infty}\mathcal N(P_{n,n'})$.\begin{lemma}\label{delta}There exists $n_1\in\mathbb N$ such that, for any $n\ge n_1$,

$\begin{array}{ll}\exists H_n\text{ infinite subset of }\{n,n+1,\dots\}\ &\mathcal N(P_{n,n'})=\delta\ (n'\in H_n)\\&\mathcal N(P_{n,n'})\le\delta\ (n'\ge\min H_n).\end{array}$
\end{lemma}

\begin{proof}[{\bf Proof}]We note first that $\mathcal N(P_{n,n'})\le\mathcal N(P_{n+1,n'})$ because, if $\mathcal Z(P_{n,n'}u_{j_1})$, $\mathcal Z(P_{n,n'}u_{j_2}),\dots$ are some distinct columns of $\mathcal Z(P_{n,n'})$, then since $P_{n,n'}=A_{n+1}P_{n+1,n'}$ the columns $\mathcal Z(P_{n+1,n'}u_{j_1}),\mathcal Z(P_{n+1,n'}u_{j_2}),\dots$ are necessarily distinct.

Consequently the sequence $n\mapsto\limsup_{n'\to\infty}\mathcal N(P_{n,n'})$ is non-decreasing. It is constant from a rank $n=n_1$ and Lemma \ref{delta} follows.\end{proof}




We chose, among the elements of $H_{n_1}$, some integers $n_2<n_3<\dots$ such that $\mathcal I(P_{n_1,n_2})=\mathcal I(P_{n_1,n_3})=\dots$. Using the last inequality of Lemma \ref{delta}, we can chose $n_2,n_3,\dots$ large enough to have $\mathcal N(P_{n_k,n_\ell})\le\delta$ when $1\le k<\ell$.

There exist $I_1,\dots,I_\delta$ (distinct), $J_1,\dots,J_\delta$ (partition of $\{1,\dots,d\}$) such that$$\textstyle\forall k>1\ \mathcal I(P_{n_1,n_k})=\cup_{h=1}^\delta (I_h\times J_h).$$At least one of the $I_h$ does not contain any $I_{h'}$ with $h'\ne h$. Since $P_{n_0,n_\ell}=P_{n_0,n_k}P_{n_k,n_\ell}$ and since $\mathcal I(P_{n_0,n_k})=\mathcal I(P_{n_0,n_\ell})$, we deduce that the matrix $P_{n_k,n_\ell}$ cannot have nonnull elements with column-index in $J_h$ and row-index in the complementary of $J_h$. So $\mathcal Z(P_{n_k,n_\ell})$ has the form $P\left(\begin{smallmatrix}B_{k,\ell}&C_{k,\ell}\\0&D_{k,\ell}\end{smallmatrix}\right)P^{-1}$ with $P$ permutation matrix, $\mathcal N(B_{k,\ell})=1$ and $\mathcal N(D_{k,\ell})=\delta-1$ (if $\delta=1$, $B_{k,\ell}$ is the matrix $\mathcal Z(P_{n_k,n_\ell})$ itself).

Assuming that the theorem is true for the sequences of matrices $\mathcal A$ such that $\delta_\mathcal A=\delta-1$, we can suppose $D_{k,\ell}=D$ constant, after replacing $(n_k)_{k\in\mathbb N}$ by a suitable subsequence.

Denoting by $B_k$ (resp. $C_k$) the matrix $B_{k,k+1}$ (resp. $C_{k,k+1}$), one has $\left(\begin{smallmatrix}B_{k,\ell}&C_{k,\ell}\\0&D\end{smallmatrix}\right)=\mathcal Z(\prod_{i=k}^{\ell-1}\left(\begin{smallmatrix}B_i&C_i\\0&D\end{smallmatrix}\right))$ and, since $\mathcal N(B_k)=1$, $B_{k,\ell}=B_k$ except if one of the $B_i$ is null. Replacing $(n_k)_{k\in\mathbb N}$ by a suitable subsequence, $B_{k,\ell}=B_k=0$ in this last case.

$C_{k,\ell}=\mathcal Z(C_kD+B_k\sum_{k<i<\ell-1}C_iD+B_kC_{\ell-1})$. Since the sequence $\ell\mapsto\mathcal Z(\sum_{k<i<\ell-1}C_i)$ is not-decreasing (the entries are not-decreasing), it converges to a $(0,1)$-matrix that we denote by $\mathcal Z(\sum_{k<i<\infty}C_i)$, by abuse of notation. Now $k\mapsto\mathcal Z(\sum_{k<i<\infty}C_i)$ is not-increasing and there exists $\kappa$ such that this sequence converges to $\mathcal Z(\sum_{\kappa<i<\infty}C_i)=:S$.

There exists $K$, infinite set of integers, such that $(B_k,C_k,C_{k-1})=\text{constant}$

$=:(B,C,C')$, so, replacing $(n_k)_{k\ge1}$ by a suitable subsequence one has

$\mathcal Z(P_{n_k,n_\ell})=P\left(\begin{smallmatrix}B&\mathcal Z(CD+BSD+BC')\\0&D\end{smallmatrix}\right)P^{-1}$ and $P_{n_k,n_\ell}\in T^1_\mathcal J$ ($1\le k<\ell$).

To prove the existence of a minimal sequence $(n_k)_{k\ge1}$ for the lexicographical order, that satisfies both conditions of Theorem \ref{triangular}, we consider more generally a set $\mathcal S\ne\emptyset$ of sequences $(n_k)_{k\ge1}$ defined by a condition $\mathcal C$, and we assume the equivalence: $(n_k)_{k\ge1}$ satisfies $\mathcal C$ if and only if the finite sequence $(n_k)_{k=1}^\ell$ satisfies $\mathcal C$ for any $\ell\ge1$. Considering the set $\mathcal S_u$ of the sequences less or equal to a fixed element $u=(n_k)_{k\ge1}\in\mathcal S$, one can easily define by induction the element of $\mathcal S$ minimal for the lexicographical order.\end{proof}

\end{document}